%% file: main.tex
\documentclass[leqno]{amsart}
\usepackage[a4paper, margin=3cm]{geometry}
\usepackage{amsthm}
\usepackage{amsmath,amssymb,mathtools,mathdots}
\usepackage{mathrsfs}
\usepackage{url}
\usepackage[all]{xy}
\usepackage{tikz}
\usepackage{tikz-cd}
\usetikzlibrary {arrows.meta}
\usepackage{comment}
\usepackage{placeins}

\newtheorem{theorem}{Theorem}[section]
\newtheorem*{theorem*}{Theorem}

\newtheorem*{lemma*}{Lemma}
\newtheorem{proposition}[theorem]{Proposition}
\newtheorem*{proposition*}{Proposition}
\newtheorem{corollary}[theorem]{Corollary}
\newtheorem*{corollary*}{Corollary}

\newtheorem*{claim*}{Claim}

\newtheorem*{fact*}{Fact}
\newtheorem{conjecture}[theorem]{Conjecture}
\newtheorem*{conjecture*}{Conjecture}
\theoremstyle{definition}
\newtheorem{definition}[theorem]{Definition}
\newtheorem*{definition*}{Definition}

\newtheorem*{example*}{Example}
\newtheorem{remark}[theorem]{Remark}
\newtheorem*{remark*}{Remark}

\newtheorem*{question*}{Question}

\newtheorem*{assumption*}{Assumption}
\numberwithin{equation}{section}
\allowdisplaybreaks[1]

\DeclareMathOperator{\ad}{ad}

\DeclareMathOperator{\id}{id}

\newcommand{\C}{{\mathbb C}}
\newcommand{\R}{{\mathbb R}}

\newcommand{\w}{{\, \wedge \,}}
\newcommand{\pa}{{\partial}}
\newcommand{\bpa}{{\overline{\partial}}}
\newcommand{\iu}{{\mathrm{i}}}
\newcommand{\br}[2]{{[\,{#1},{#2}\,]}}
\newcommand{\conj}[1]{{\overline{#1}}}

\newcommand{\bmu}{{\overline{\mu}}}
\newcommand{\wed}[1]{{\epsilon_{#1}}}
\newcommand{\weda}[1]{{\epsilon^*_{#1}}}

\usepackage{multirow}

\address{graduate school of mathematical sciences, the university of tokyo, 3-8-1 komaba, meguro, tokyo 153-8914, japan.}
\email{shuho@ms.u-tokyo.ac.jp}
\begin{document}
\title{The hard Lefschetz duality for locally conformally almost K\"{a}hler manifolds}
\author{Shuho Kanda}
\date{}
\maketitle
\begin{abstract}
    \input{abstruct}
\end{abstract}

\input{Introduction}

\input{Acknowledgements}

\input{Preliminaries}

\input{KidforLCaK}

\input{HLC}

\input{Examples}

\FloatBarrier
\bibliographystyle{amsalpha}
\bibliography{M_paper}
\end{document}

%% file: abstruct.tex
We prove the hard Lefschetz duality for 
locally conformally almost K\"{a}hler manifolds. 
This is a generalization of that 
for almost K\"{a}hler manifolds studied by Cirici and Wilson. 
We generalize the K\"{a}hler identities to prove the duality. 
Based on the result, 
we introduce the hard Lefschetz condition for 
locally conformally symplectic manifolds. 
As examples, 
we give solvmanifolds which do not satisfy the 
hard Lefschetz condition. 

%% file: Introduction.tex
\section{Introduction}

\subsection{Background and motivation}
    K\"{a}hler manifolds are a special class of symplectic manifolds. 
    The first example of a symplectic manifold which does not admit a K\"{a}hler structure 
    was constructed by William Thurston in 1976 \cite{MR0402764}. 
    Since then,  
    methods for determining whether a symplectic manifold admits a K\"{a}hler structure 
    have been actively studied.  
    
    One of the most basic methods is to see the ring structure of 
    the de Rham cohomology. 
    We focus here on the \emph{hard Lefschetz condition} (HLC, for short). 
    Let $(M,\omega)$ be a $2n$-dimensional symplectic manifold. 
    We say that $(M,\omega)$ satisfies the HLC 
    if the $(n-j)$-th exterior power of the symplectic form $\omega^{n-j}$ 
    induces an isomorphism between the cohomologies 
    $H^j(M)$ and $H^{2n-j}(M)$
    for all $1 \le j \le n$. 
    It is known that all compact K\"{a}hler manifolds satisfy the HLC 
    (see \cite{MR2093043}, \cite{MR0507725} ). 
    This means that a compact symplectic manifold $(M,\omega)$
    which does not satisfy the HLC 
    never admits a K\"{a}hler structure
    $(M,J,g)$ such that the associated fundamental form 
    $g(J\cdot,\cdot)$ is $\omega$. 

    We can see how a compact symplectic manifold fails to 
    satisfy the HLC by using an additional structure of it. 
    Every compact symplectic manifold $(M,\omega)$ 
    has at least one almost K\"{a}hler structure 
    $(M,J,g)$ such that the associated fundamental form is $\omega$
    (we call the structure K\"{a}hler if $J$ is integrable). 
    Define the space of harmonic forms of bidegree $(p,q)$ by letting
    \[
    \mathcal{H}^{p,q} \coloneqq \ker(\Delta_d) \cap \mathcal{A}^{p,q}, 
    \]
    where $\mathcal{A}^{p,q}$ denotes the space of $(p,q)$-forms. 
    In \cite{MR4110721}, 
    Cirici and Wilson studied the properties of 
    the space $\mathcal{H}^{p,q}$. 
    According to the generalized hard Lefschetz duality, 
    one of the results in \cite{MR4110721}, 
    it is shown that the maps 
    \begin{equation}\label{Ldual}
    L^{n-(p+q)} \colon \mathcal{H}^{p,q} 
    \stackrel{\sim}{\longrightarrow} \mathcal{H}^{n-q,n-p}
    \end{equation}
    are isomorphisms for all $p,q$, 
    where $L$ denotes an operator defined by wedging a form with $\omega$. 
    In general, 
    the inclusion 
    \begin{equation}\label{inc}
        \bigoplus_{p+q=j} \mathcal{H}^{p,q} \subset \mathcal{H}^j
    \end{equation}
    is proper, 
    where $\mathcal{H}^j$ denotes 
    a space of harmonic $j$-forms. 
    If $(M,J,g)$ is K\"{a}hler, 
    i.e. $J$ is integrable, 
    it turns out that the inclusion (\ref{inc}) is equal by 
    the K\"{a}hler identities. 
    Hence all K\"{a}hler manifolds satisfy the HLC. 
    In this sense, 
    we can say that the gap of the inclusion (\ref{inc}) measures to 
    what extent is the HLC holds. 

    We consider a class of \emph{locally conformally symplectic} (LCS) manifolds, 
    which is a generalization of symplectic manifolds. 
    Let $M$ be a manifold and $\omega$ be a nondegenerate $2$-form on $M$. 
    We call the pair $(M,\omega)$ LCS if, 
    for all points of $M$, 
    there exists a neighborhood $U$ of the point and 
    real-valued function $f$ on $U$ such that 
    $(U,e^{-f}\omega)$ is symplectic. 
    This condition is equivalent to saying that 
    there exists a closed $1$-form $\theta$ on $M$, 
    which is called the \emph{Lee form}, 
    such that $d\omega = \theta \w \omega$. 
    As the similar generalization of a K\"{a}hler manifold, 
    a \emph{locally conformally K\"{a}hler} (LCK) manifold is defined. 
    Typical examples of LCK manifolds are 
    Hopf manifolds, Inoue surfaces, and Thurston's manifolds,  
    which are basic examples of complex manifolds which do not admit any K\"{a}hler structures (see \cite{MR1481969}). 
    The class of LCS and LCK manifolds has attracted 
    much interests among geometers in recent years. 

    As in the case of symplectic manifolds, 
    it is natural to ask whether an LCS manifold admits 
    an LCK structure. 
    This existence problem is explicitly proposed by Ornea and Verbitsky in
    \cite{MR2796645}. 
    A few examples are discovered only in \cite{MR3724467} and \cite{MR3473558}. 
    One of the difficulties of the problem is that,  
    in contrast to the case of K\"{a}hler structures, 
    there are no known topological 
    obstructions for the existence of LCK structures. 
    See \cite{ornea2023principles} for details. 

    One of the clues to the problem is the HLC. 
    In recent years, 
    generalizations of the HLC to LCS manifolds have been studied. 
    In \cite{MR3746516}, 
    a generalization of the HLC for LCS is proposed by using 
    the \emph{twisted de Rham cohomology}
    (see also \cite{MR3318013}, 
    where the symplectic cohomology on LCS manifolds is studied 
    by using the twisted differential). 

    There are two main purpose of this paper. 
    First, 
    to reformulate the definition of the HLC in \cite{MR3746516}. 
    Second, 
    as a generalization of the results in \cite{MR4110721}, 
    to develop the hard Lefschetz duality for  
    \emph{locally conformally almost K\"{a}hler} (LCaK) 
    manifolds, 
    which are non-integral version of LCK manifolds. 

\subsection{The results of this paper}
    We formulate the HLC for LCS manifolds as follows: 
    \begin{definition}[The hard Lefschetz condition for locally conformally symplectic manifolds, Definition \ref{HLCforLCS}]
        Let $(M,\omega)$ be a $2n$-dimensional 
        compact locally conformally symplectic manifold with the Lee form $\theta$. 
        $(M,\omega)$ satisfies the
        \emph{hard Lefschetz condition} 
        if the maps
        \begin{equation}\label{Liso}
        [L]^{n-j} \colon H^j_{-(n-j)/2}(M) 
        \to H^{2n-j}_{(n-j)/2}(M)
        \end{equation}
        are isomorphisms for all $1 \le j \le n$. 
    \end{definition}

    In the above definition, 
    we write $H^j_k(M)$ for the cohomology group of a complex 
    composed of differential forms and 
    the \emph{twisted differential} $d_{k}$ defined by 
    $d_{k}\alpha \coloneqq d\alpha-k\theta \w \alpha$ for $k \in \R$. 
    Using this cohomology, 
    which is called the twisted de Rham cohomology, 
    we can define the \emph{Lefschetz map} $[L]^{n-j}$ as follows: 
    \begin{equation}\label{AOT}
    [L]^{n-j} \colon H^j_{k}(M) 
    \to H^{2n-j}_{k+n-j}(M), 
    \quad
    [\alpha] \mapsto [\omega^{n-j} \w \alpha]. 
    \end{equation}
    According to the results in \cite{MR4376841}, 
    it turns out that 
    the space $H^j_{t}(M)$ is isomorphic to the space $H^{2n-j}_{-t}(M)$ 
    for all $t \in \R$ by the Hodge theory. 
    This is why we set the parameter $k$ in (\ref{AOT}) to $-(n-j)/2$ so that 
    the relation $k=-(k+n-j)$ holds, 
    and adopt the above definition. 

    We write $\Lambda$ for the adjoint operator of $L$. 
    Denote by $\mathcal{H}^{p,q}_k$ the space of $d_k$-harmonic forms of 
    bidegree $(p,q)$. 
    The space $\mathcal{H}^{j}_{k,\mathcal{J}}$ 
    is a space of $d_k$-harmonic forms $\alpha$ such that 
    $\mathcal{J}\alpha$ is also $d_k$-harmonic, 
    where $\mathcal{J}$ is an operator induced by $J$. 

    Using the above spaces and the formulation of the HLC, 
    we generalize the hard Lefschetz duality (\ref{Ldual}) as follows: 

    \begin{theorem}
    [The hard Lefschetz duality for locally conformally almost K\"{a}hler manifolds, 
    Theorem \ref{main}]
        Let $(M,J,g)$ be a $2n$-dimensional 
        compact locally conformally almost K\"{a}hler manifold with the Lee form $\theta$. 
        Then the operators 
        $\{ L, \Lambda, H\coloneqq \br{L}{\Lambda}\}$ 
        define finite 
        dimensional representations of $\mathfrak{sl}(2)$ on 
        \begin{equation}\label{slharm0}
        \bigoplus_{p,q \ge 0} \mathcal{H}^{p,q}_{-(n-(p+q))/2} 
        \quad \text{and} \quad
        \bigoplus_{j \ge 0} \mathcal{H}^j_{-(n-j)/2, \mathcal{J}}. 
        \end{equation}
        Furthermore, 
        for all $0 \le p,q$ such that $j=p+q\le n$, 
        \[
        L^{n-j} \colon \mathcal{H}^{p,q}_{-(n-j)/2} 
        \, \stackrel{\sim}{\longrightarrow} \, 
        \mathcal{H}^{n-q,n-p}_{(n-j)/2}
        \]
        and 
        \[
        L^{n-j} \colon \mathcal{H}^j_{-(n-j)/2, \mathcal{J}} 
        \, \stackrel{\sim}{\longrightarrow} \,  
        \mathcal{H}^{2n-j}_{(n-j)/2, \mathcal{J}}
        \]
        are isomorphisms. 
    \end{theorem}

    To prove this theorem, 
    we establish the K\"{a}hler identities for LCaK manifolds
    (Proposition \ref{KidforLCaK}). 
    Furthermore, 
    we can see the importance of the relation $k=-(k+n-j)$
    in the proof. 
    
    Using this theorem, 
    we can see how the ranks of the Lefschetz maps in 
    (\ref{Liso}) drop. 
    As examples, 
    we give a family of LCaK manifolds which do not satisfy the HLC 
    in Section \ref{secExamples}. 
    For each integer $m \ge 2$, 
    we construct a $(2m+2)$-dimensional compact solvable Lie group $M_m$ and 
    a left-invariant LCaK structure of it. 
    The manifolds are the same as 
    those constructed in Section $4$ of \cite{MR3763412}. 
    On manifolds $M_2$, $M_3$, and $M_4$, 
    we give the basis of the spaces in (\ref{Liso}) and (\ref{slharm0}) 
    and compute the Lefschetz maps. 
    
    Finally, 
    we propose a conjecture on the HLC for LCK manifolds. 
    Recall that the HLC for symplectic manifolds 
    is always satisfied on K\"{a}hler manifolds. 
    According to this fact, 
    the following conjecture is naturally suggested: 
    
    \begin{conjecture}[Conjecture \ref{HLTforLCK}]
        Let $(M,J,g)$ be a $2n$-dimensional compact 
        LCK manifold with the Lee form $\theta$. 
        Then, 
        $(M,J,g)$ satisfies the HLC. 
    \end{conjecture}
    
    The above conjecture is still open. 

\subsection{Organization of this paper}
    This paper is organized as follows: 
    \begin{itemize}
        \item In Section \ref{secPreliminaries} 
        we collect preliminaries on LCaK structures, 
        the twisted cohomology, 
        operators on almost complex manifolds, 
        and the Lefschetz decomposition. 
        \item In Section \ref{secKidforLCaK} 
        we write down K\"{a}hler identities for LCaK manifolds 
        and prove them (Proposition \ref{KidforLCaK}). 
        \item In Section \ref{secHLC}
        we define some subspaces of twisted harmonic forms. 
        Then we prove the main theorem (Theorem \ref{main}) and 
        define the HLC for LCS manifolds (Definition \ref{HLCforLCS}). 
        \item In Section \ref{secExamples} 
        we give examples of LCaK manifolds which do not satisfy the HLC 
        and compute the spaces defined in the previous sections 
        with their basises. 
    \end{itemize}

%% file: Acknowledgements.tex
\section*{Acknowledgements}
I am deeply grateful to my supervisor Prof. Kengo Hirachi 
for his enormous support and helpful advice. 
I am also grateful to Prof. Hisashi Kasuya 
for helpful discussions. 
Lastly, 
I also would like to thank Jin Miyazama for 
several helpful comments. 
This research was supported by Forefront Physics and Mathematics Program to
Drive Transformation (FoPM), a World-leading Innovative Graduate Study (WINGS)
Program, the University of Tokyo. 

%% file: Preliminaries.tex
\section{Preliminaries}\label{secPreliminaries}
    In this section we review the basics 
    of locally conformally almost K\"{a}hler geometry, 
    the twisted cohomology, 
    and the Lefschetz decomposition. 
    Throughout this paper, 
    we only consider smooth connected manifolds 
    without boundary which have even real dimension $2n \ge 4$. 
    We always assume that functions and differential forms 
    are smooth. 
    Furthermore, 
    the vector spaces defined in this paper, 
    such as the space of harmonic forms, 
    the de Rham cohomology group, 
    etc, 
    are always assumed to be complex vector spaces, 
    not real vector spaces. 
    
\subsection{Locally conformally almost K\"{a}hler manifolds}
    Let $M$ be a smooth manifold of dimension $2n \ge 4$ and
    $\omega$ be a nondegenerate $2$-form on $M$, 
    i.e. 
    $\omega^n \neq 0$ at all points of $M$. 
    The form $\omega$ is said to be \emph{conformally symplectic} (CS, for short) on $M$ 
    if there exists a real-valued function $f$ 
    such that $e^{-f} \omega$ is symplectic on $M$, 
    i.e. 
    $d(e^{-f} \omega)=0$. 
    The form $\omega$ is said to be \emph{locally conformally symplectic} (LCS) on $M$ 
    if there exists an open covering $\{U_i\}_{i \in I}$ of $M$ 
    such that $\omega|_{U_i}$ are CS on $U_i$ for all $i \in I$. 
    Let $f_i$ be a real-valued function on $U_i$ 
    such that $e^{-f_i} \omega$ is symplectic on $U_i$. 
    Then the condition that the form $e^{-f_i} \omega$ is symplectic is equivalent to 
    $d\omega=df_i \w \omega$. 
    Since $\omega$ is nondegenerate, 
    $df_i \w \omega=df_j \w \omega$ implies 
    $df_i=df_j$ on $U_i \cap U_j$. 
    So the $1$-forms $\{df_i\}_{i \in I}$ are glued together. 
    Thus we obtain the following proposition 
    which makes the above definition simpler: 
    
    \begin{proposition}
        A nondegenerate $2$-form $\omega$ on $M$ is LCS if and only if 
        there exists a closed $1$-form $\theta$ on $M$ 
        such that $d\omega = \theta \w \omega$. 
        Furthermore, 
        $\omega$ is CS if and only if $\theta$ is exact. 
    \end{proposition}
    
    The $1$-form $\theta$ is determined by $\omega$, 
    and we call it the \emph{Lee form} of $\omega$. 
    In these situations we call the couple 
    $(M,\omega)$ an \emph{LCS manifold} (\emph{CS manifold}) with the Lee form $\theta$. 

    If $M$ admits an almost complex structure 
    $J$ and a hermitian metric $g$ such that
    the associated fundamental form 
    $\omega(\cdot \, , \cdot) \coloneqq g(J\cdot \, ,\cdot)$ 
    is LCS (CS), 
    we call the triple $(M,J,g)$ a 
    \emph{locally conformally (conformally) almost K\"{a}hler} (LCaK (CaK)). 
    Furthermore, 
    we call the triple $(M,J,g)$ a 
    \emph{locally conformally (conformally) K\"{a}hler} (LCK (CK)) 
    if the almost complex structure $J$ is integrable.  
    The relation among the classes of manifolds 
    is illustrated in Figure \ref{class}. 

    \begin{figure}[h]
    \centering
    \caption{} \label{class}
    \begin{tikzpicture}
        \draw (0,0) rectangle (5.5,3);
        \draw (0.4,0.5) rectangle (3,1.9);
        \draw [-{Classical TikZ Rightarrow[length=2mm]}] (6,1.5) -- (8.5,1.5);
        \draw (9,0) rectangle (14.5,3);
        \draw (9.4,0.5) rectangle (12,1.9);
        \draw (12.6,1.2) circle [x radius=1.8, y radius=0.6];
        \node at (2.75,3.2) {LCS};
        \node at (1.6,2.1) {CS};
        \node at (11.75,3.2) {LCaK};
        \node at (10.6,2.1) {CaK};
        \node at (12.8,2) {LCK};
        \node at (11.5,1.2) {CK};
        \node at (7.1,2.3) [align=center] 
        {add \\ almost hermitian \\ structures};
    \end{tikzpicture}
    \end{figure}

    \vspace{5mm}
    
    Lastly, 
    we note that for every nondegenerate $2$-forms $\omega$ on $M$ 
    there exists an almost complex structure $J$ and a hermitian metric $g$ such that 
    the fundamental form of $(M,J,g)$ coincides with $\omega$. 

\subsection{The twisted de Rham cohomology}
    The twisted de Rham cohomology is also called 
    the Morse-Novikov cohomology, 
    the Lichnerowicz cohomology, 
    or the adapted cohomology. 
    Its study was initiated by Novikov 
    \cite{Novikov1982}, 
    and developed by 
    Guedira and Lichnerowicz 
    \cite{Gudira1984GomtrieDA}. 
    For an introduction to the twisted de Rham cohomology, 
    see \cite{MR2501742}. 
    Here we give the definitions and some basic properties of them. 
    
    Let $M$ be a $2n$-dimensional manifold. 
    Denote by $\mathcal{A}^j$ the space of complex-valued differential $j$-forms on $M$, 
    and $H^j(M)$ the $j$-th cohomology group 
    of the complex $(\mathcal{A}^{\bullet}\, , d)$,  
    where $d$ is the exterior derivative. 
    Let $\theta$ be a closed $1$-form. 
    We define the \emph{twisted differential} $d_{\theta}$ as 
    \[
    d_{\theta} \coloneqq d - \wed{\theta} 
    \colon \mathcal{A}^{\bullet} \to \mathcal{A}^{\bullet+1},
    \]
    where $\wed{\alpha}$ denotes an operator defined by 
    $\cdot \mapsto \alpha \w \cdot$ for a form $\alpha$. 
    Similar to the usual differential, 
    we have $d_{\theta} d_{\theta}=0$. 
    Denote by $H^j_{\theta}(M)$ the $j$-th cohomology group 
    of the complex $(\mathcal{A}^{\bullet}\, , d_{\theta})$, 
    and we call it the 
    \emph{twisted de Rham cohomology} or simply the 
    \emph{twisted cohomology}. 
    The cohomology group $H^j_{\theta}(M)$ depends only on the cohomology class
    of $[\theta]$ in $H^1(M)$. 
    Indeed, 
    if $\theta^{'}=\theta+df$, 
    where $f$ is a function on $M$, 
    there exists an isomorphism of complexes:  
    \[
    e^f \colon (\mathcal{A}^{\bullet}\, , d_{\theta}) 
    \stackrel{\sim}{\longrightarrow} 
    (\mathcal{A}^{\bullet}\, , d_{\theta^{'}}).
    \]
    In particular, 
    the twisted cohomology 
    $H^j_{\theta}(M)$ is isomorphic to the 
    de Rham cohomology $H^j(M)$ if
    $\theta$ is exact. 

    Let $k \in \R$ be a real number. 
    We write $d_{k\theta}$ and $H^j_{k\theta}(M)$ as 
    $d_k$ and $H^j_k(M)$ respectively 
    when $\theta$ is clear. 
    Using these notation, 
    we have the following Leibniz rule:  
    \begin{equation}\label{Leibniz}
        d_{k+l}(\alpha \w \beta)=d_k \alpha \w \beta + 
        (-1)^{\deg{\alpha}} \alpha \w d_l \beta 
    \end{equation}
    for all $k,l \in \R$. 
    From this formula, 
    we have the following map induced by the wedge product: 
    \begin{equation}\label{wedge}
        \w \colon H^{j_1}_{k_1}(M) \times H^{j_2}_{k_2}(M) 
        \to H^{j_1+j_2}_{k_1+k_2}(M). 
    \end{equation}

    We consider the case where $M$ is compact, 
    orientable, and endowed with a Riemannian metric $g$. 
    Denote by $d_k^*$ the adjoint operator of $d_k$ 
    with respect to the metric $g$, 
    and we define the \emph{twisted laplacian}
    $\Delta_k \coloneqq d^*_kd_k+d_kd^*_k$. 
    Since the twisted laplacian $\Delta_k$ is elliptic, 
    the space of differential forms 
    $\mathcal{A}^{\bullet}$ has 
    the following Hodge decomposition: 
    \[
    \mathcal{A}^{\bullet} = 
    \mathcal{H}^{\bullet}_k \oplus
    d_k\mathcal{A}^{\bullet-1} \oplus
    d^*_k\mathcal{A}^{\bullet+1}, 
    \]
    where $\mathcal{H}^{\bullet}_k$ denotes the kernel of 
    $\Delta_k$, 
    which is called the space of \emph{twisted harmonic forms}. 
    As a consequence, 
    the natural map 
    $\mathcal{H}^{\bullet}_k \to H^{\bullet}_k(M)$
    defined by 
    $\alpha \mapsto [\alpha]$
    is an isomorphism. 
    We also have $\dim H^{\bullet}_k(M) < \infty$. 
    The Hodge duality for the twisted cohomologies is 
    stated as follows:
    \begin{proposition}
    [\cite{MR4376841}]
    \label{Hodge}
        In the above situation, 
        we have 
        \[
        *\Delta_k=\Delta_{-k}*
        \]
        for all $k \in \R$. 
        In particular, 
        the maps
        \[
        * \colon \mathcal{H}^j_k 
        \, \stackrel{\sim}{\longrightarrow} \, 
        \mathcal{H}^{2n-j}_{-k}
        \]
        are defined and isomorphisms for 
        all $0 \le j \le 2n$. 
   \end{proposition}

\subsection{Operators associated with an almost complex structure}

    Let $(M,J)$ be an almost complex manifold, 
    and
    \[
    \mathcal{A}^j 
    = \bigoplus_{p+q=j} \mathcal{A}^{p,q}
    \]
    be the bigraded decomposition of the space of $j$-forms 
    $\mathcal{A}^j$ with respect to $J$. 
    We fix a closed $1$-form $\theta$ on $M$ and $k \in \R$. 
    The twisted differential $d_k$ splits into 
    four components: 
    \[
    d_k=\pa_k+\bpa_k+\mu+\bmu, 
    \]
    where the bidegrees of the operators 
    $\pa_k$, $\bpa_k$, $\mu$ and $\bmu$ are 
    $(1,0)$, $(0,1)$, $(2,-1)$, and $(-1,2)$ respectively. 
    Then we have 
    $\pa_k=\pa-k\wed{\theta^{1,0}}$ and $\bpa_k=\bpa-k\wed{\theta^{0,1}}$,  
    where $\theta^{1,0}$ and $\theta^{0,1}$ 
    denote the $(1,0)$-part and $(0,1)$-part 
    of $\theta$ respectively. 
    Since the operator $\mu$ acts as zero on $0$-forms, 
    $\mu$ and $\bmu$ are linear over functions. 
    It is well known that $J$ is integrable if and only if $\mu=0$. 
    Splitting the equation $d_kd_k=0$ into bidegrees, 
    we obtain the following equations: 
    \begin{align*}
        \pa_k\pa_k &= -\mu\bpa_k-\bpa_k\mu, \\
        \pa_k\bpa_k+\bpa_k\pa_k &= -\mu\bmu-\bmu\mu, \\
        \mu\pa_k &= -\pa_k\mu, \\
        \mu\mu &= 0. 
    \end{align*}

    The almost complex structure $J$ induces the real isomorphism 
    \[
    \mathcal{J} \coloneqq \sum_{p+q=j} \iu^{p-q} \, \Pi^{p,q} 
    \colon \mathcal{A}^j \to \mathcal{A}^j, 
    \]
    where $\Pi^{p,q}$ denotes the projection 
    $\mathcal{A}^j \to \mathcal{A}^{p,q}$, 
    and $\iu$ denotes the imaginary unit. 
    Using this operator, 
    we define a real operator $d^c_k$ as 
    \[
    d^c_k \coloneqq \mathcal{J}^{-1}d_k\mathcal{J}. 
    \]
    By straightforward calculations, 
    we have
    \[
    d^c_k = -\iu(\pa_k-\bpa_k)+\iu(\mu-\bmu). 
    \]
    The \emph{anti-Lee form} $\theta^c$ is defined by 
    \[
    \theta^c \coloneqq -\mathcal{J}\theta =-\theta \circ J
    = -\iu(\theta^{1,0}-\theta^{0,1}). 
    \]
    Using this, 
    we can write 
    \[
    \mathcal{J}^{-1}\wed{\theta}\mathcal{J}=\wed{\theta^c}. 
    \]
    So we have 
    \begin{equation}\label{d^c_k}
        d^c_k = d^c - k \wed{\theta^c}.
    \end{equation}

\subsection{The Lefschetz decomposition on the space of forms}
\label{secLef}

    We review some basic facts about 
    the Lefschetz decomposition on the space of forms. 
    See e.g. 
    \cite{MR2093043}, 
    \cite{MR0507725}
    for details. 
    
    We consider an almost hermitian manifold $(M,J,g)$ with 
    the associated fundamental form $\omega$. 
    The \emph{Lefschetz map} is an operator defined by 
    wedging a form with $\omega$: 
    \[
    L \coloneqq \wed{\omega} \colon 
    \mathcal{A}^{p,q} \to \mathcal{A}^{p+1,q+1}, 
    \]
    and its formal adjoint operator $\Lambda$ is 
    \[
    \Lambda \coloneqq (-1)^{p+q}*L* \colon 
    \mathcal{A}^{p,q} \to \mathcal{A}^{p-1,q-1}. 
    \]
    It is well known that the operators
    $\{L, \Lambda, H \coloneqq \br{L}{\Lambda} \}$
    define a representation of $\mathfrak{sl}(2)$ on 
    \[
    \mathcal{A}^{\bullet}=\bigoplus_{j} \mathcal{A}^j. 
    \]
    The $\mathfrak{sl}(2)$-representation induces the 
    \emph{Lefschetz decomposition} 
    \[
    \mathcal{A}^{p,q}= 
    \bigoplus_{j \ge 0} L^j \mathcal{P}^{p-j,q-j}, 
    \]
    where $\mathcal{P}^{p-j,q-j}$ denotes 
    the space $\ker(\Lambda) \cap \mathcal{A}^{p-j,q-j}$. 
    From the decomposition, 
    the maps 
    \[
    L^j \colon \mathcal{A}^{p,q} \to \mathcal{A}^{p+j,q+j}
    \]
    turn out to be injective for all 
    $0 \le j \le n-(p+q)$. 
    In this paper, 
    we call the composite $L^j$ also the Lefschetz map.

%% file: KidforLCaK.tex
\section{The K\"{a}hler identities for 
locally conformally almost K\"{a}hler manifolds}
\label{secKidforLCaK}

    Let $(M,J,g)$ be an almost hermitian manifold 
    and $\omega$ be the associated fundamental form. 
    In the case where $(M,J,g)$ is almost K\"{a}hler, 
    i.e. $\omega$ is closed, 
    there are remarkable  relations among operators which 
    are the so-called \emph{K\"{a}hler identities} 
    (the term ``K\"{a}hler identities'' is often used for  
    the identities in the case where 
    $J$ is integrable, however, 
    here we cite Weil's classical result 
    \cite{MR0111056} for the non-integrable cases. 
    See also \cite{MR4110721}. 
    The references for the integrable cases, 
    see e.g. 
    \cite{MR2093043}, \cite{MR0507725}). 

    In what follows, 
    we write $\delta^*$ for the formal adjoint operator of 
    an operator $\delta$. 
    It is well known that $d^*=-*d*$. 
    Splitting the formula into bidegree, 
    we have
    \[
    \bmu^*=-*\mu* \quad \text{and} \quad \bpa^*=-*\pa*. 
    \]

    \begin{proposition}
    [The K\"{a}hler identities for almost 
    K\"{a}hler manifolds,
    \cite{MR0111056}]
    \label{Kid}
        Let $(M,J,g)$ be an almost K\"{a}hler manifold with 
        the associated fundamental form $\omega$. 
        Then the following identities hold:  
        \input{eq_Kid}
    \end{proposition}
    We generalize these identities to LCaK manifolds.  

    \begin{proposition}
    [The K\"{a}hler identities for LCaK manifolds]
    \label{KidforLCaK}
        Let $(M,J,g)$ be a $2n$-dimensional 
        LCaK manifold with the associated 
        fundamental form $\omega$
        and the Lee form $\theta$. 
        Then the following identities hold as operators 
        acting on $\mathcal{A}^j$:  
        \input{eq_KidforLCaK}
    \end{proposition}
    
    The K\"{a}hler identities for LCK manifolds, 
    which are the special case of the above proposition, 
    are obtained in \cite{MR3530921}. 
    Here we give a proof using conformal changes of 
    the operators. 
    Before the proof, 
    we give some definitions and basic facts. 
    Let $f$ be a real-valued function. 
    We consider the conformal change 
    $\tilde{g} \coloneqq e^{-f}g$ 
    and the fundamental form
    $\tilde{\omega} \coloneqq e^{-f}\omega$. 
    Denote by $\tilde{*}$ 
    the Hodge star operator with respect to the metric $\tilde{g}$. 
    It is straightforward to check that the equality 
    \begin{equation}\label{conf*}
        \tilde{*}=e^{-(n-j)f}*
    \end{equation} 
    holds as an operator acting on $\mathcal{A}^j$.
    We also define the operator 
    $d^{\tilde{*}} \coloneqq -\tilde{*}d\tilde{*}$ 
    which is the formal adjoint of $d$ with respect to $\tilde{g}$. 
    For a $1$-form $\alpha$ we define the operator 
    $\weda{\alpha} \coloneqq *\wed{\alpha}*$ 
    which is the formal adjoint operator of $\wed{\alpha}$ with 
    respect to $g$. 
    Taking the adjoints of the equality
    \[
    \br{d}{e^{-f}}=-e^{-f} \wed{df}, 
    \]
    we have
    \begin{equation}\label{[d,e]}
        \br{d^*}{e^{-f}}=e^{-f}\weda{df}. 
    \end{equation}
    The operator $\weda{\theta}$ can be written as  
    \[
    \weda{\theta} = \iota_{\theta^{\sharp}}, 
    \]
    where $\sharp$ denotes the musical isomorphism, 
    i.e. 
    $\theta^{\sharp}$ is the vector field determined by the 
    relation $g(\theta^{\sharp},\cdot)=\theta(\cdot)$ 
    (see \cite{MR3887684}). 
    Using this, 
    we have 
    \[
    \br{L}{\weda{\theta}} = \br{L}{\iota_{\theta^{\sharp}}}
    = -\wed{(\iota_{\theta^{\sharp}}\omega)}. 
    \]
    Combining this with the equation 
    \[
    \iota_{\theta^{\sharp}}\omega
    =\omega(\theta^{\sharp}, \cdot) 
    =-g(\theta^{\sharp},J\cdot)
    =-\theta\circ J
    =\theta^c, 
    \]
    we have the relation 
    \begin{equation}\label{[L,w]}
        \br{L}{\weda{\theta}}=-\wed{\theta^c}. 
    \end{equation}

    \begin{proof}[Proof of Proposition \ref{KidforLCaK}]
        The identities (i)*,(ii)*,(iii)* and (iv)* are obtained immediately 
        by taking adjoints of the identities (i),(ii),(iii) and (iv) respectively. 
        The first identity in (i) follows from $d_1 \omega =0$ and 
        the Leibniz rule (\ref{Leibniz}). 
        By splitting this identity into bidegrees,  
        we obtain (i) and (iii). 
        Similarly, 
        we obtain (ii) and (iv) by splitting the 
        first identity in (ii). 
        Therefore, 
        it is sufficient to prove the first identity in (ii). 
 
        Since $(M,J,g)$ is an LCaK manifold, 
        we can locally take a real-valued function $f$ such that 
        the metric $\tilde{g}=e^{-f}g$ endows $(M,J)$ with the K\"{a}hler structure. 
        The associated fundamental form of $(M,J,\tilde{g})$ 
        is $\tilde{\omega}=e^{-f}\omega$, 
        and we have $\theta = df$ by the uniqueness of the Lee form. 
        From Proposition \ref{Kid}, 
        we obtain 
        \[
        [d^{\tilde{*}},\tilde{L}]=-d^c, 
        \]
        where $\tilde{L} \coloneqq \wed{\tilde{\omega}}$. 
        From the equation (\ref{conf*}), 
        $d^{\tilde{*}}$ can be calculated as follows: 
        \begin{align*}
                d^{\tilde{*}}
                &= -\tilde{*}d\tilde{*} \\
                &= -\tilde{*}d(e^{-(n-j)f}*) \\
                &= -\tilde{*}e^{-(n-j)f}(-(n-j)\wed{\theta}*+d*) \\
                &= -e^{(n-(2n-j+1))}*e^{-(n-j)f}(-(n-j)\wed{\theta}*+d*) \\
                &= e^f(d^*+(n-j)\weda{\theta}). 
        \end{align*}
        Using this and the equation (\ref{[d,e]}), 
        we obtain 
        \begin{align*}
            d^{\tilde{*}} \tilde{L} 
            &= e^f(d^*(e^{-f}L)+(n-(j+2))\weda{\theta}e^{-f}L) \\
            &= e^f[d^*,e^{-f}]L+d^*L+(n-j-2)\weda{\theta}L \\
            &= d^*L+(n-j-1)\weda{\theta}L
        \end{align*}
        and 
        \begin{align*}
            \tilde{L}d^{\tilde{*}} 
            &= Ld^*+(n-j)L\weda{\theta}. 
        \end{align*}
        Hence we have 
        \begin{align*}
            -d^c 
            &= [d^{\tilde{*}},\tilde{L}] \\
            &= [d^*,L]+(n-j-1)\weda{\theta}L -(n-j)L\weda{\theta}. 
        \end{align*}
        Therefore, we obtain 
        \begin{align*}
            d^*_{k+1}L-Ld^*_{k} 
            &= (d^*-(k+1)\weda{\theta})L-L(d^*-k\weda{\theta}) \\
            &= [d^*,L]-(k+1)\weda{\theta}L+kL\weda{\theta} \\
            &= (-d^c-(n-j-1)\weda{\theta}L+(n-j)L\weda{\theta}) 
            -(k+1)\weda{\theta}L+kL\weda{\theta} \\
            &= -(d^c+(-n-k+j)[L,\weda{\theta}]). 
        \end{align*}
        From the equations (\ref{d^c_k}) and (\ref{[L,w]}), 
        we have 
        \begin{align*}
            d^c+(-n-k+j)[L,\weda{\theta}]
            &= d^c-(-n-k+j)\wed{\theta^c} \\
            &= d^c_{-n-k+j}. 
        \end{align*}
        This proves the proposition. 
    \end{proof}

%% file: eq_Kid.tex
\begin{align*}
    \textnormal{(i)} \quad &\br{d}{L}=\br{\pa}{L}=\br{\bpa}{L}=0, &
    \textnormal{(i)*} \quad &\br{d^*}{\Lambda}=\br{\pa^*}{\Lambda}=\br{\bpa^*}{\Lambda}=0, \\
    \textnormal{(ii)} \quad &\br{d^*}{L}= -d^c, & 
    \textnormal{(ii)*} \quad &\br{d}{\Lambda}= -d^{c*}, \\
    &\br{\pa^*}{L}= -\iu\bpa, & 
    &\br{\pa}{\Lambda}= -\iu\bpa^*, \\ 
    &\br{\bpa^*}{L}= \iu\pa, & 
    &\br{\bpa}{\Lambda}= \iu\pa^*, \\
    \textnormal{(iii)} \quad &\br{\mu}{L}=\br{\bmu}{L}=0, & 
    \textnormal{(iii)*} \quad &\br{\mu^*}{\Lambda}=\br{\bmu^*}{\Lambda}=0, \\
    \textnormal{(iv)} \quad &\br{\mu^*}{L}= \iu\bmu, & 
    \textnormal{(iv)*} \quad &\br{\mu}{\Lambda}= \iu\bmu^*, \\ 
    &\br{\bmu^*}{L}= -\iu\mu, & 
    &\br{\bmu}{\Lambda}= -\iu\mu^*. 
\end{align*}

%% file: eq_KidforLCaK.tex
\begin{align*}
    \textnormal{(i)} \quad &d_{k+1}L-Ld_{k}=0, & 
    \textnormal{(i)*} \quad &d^*_{k}\Lambda-\Lambda d^*_{k+1}=0, \\
    &\pa_{k+1}L-L\pa_{k}=0, & 
    &\pa^*_{k}\Lambda-\Lambda\pa^*_{k+1}=0, \\
    &\bpa_{k+1}L-L\bpa_{k}=0, & 
    &\bpa^*_{k}\Lambda-\Lambda\bpa^*_{k+1}=0, \\ 
    \textnormal{(ii)} \quad &d^*_{k+1}L-Ld^*_{k}= -d^c_{-n-k+j}, & 
    \textnormal{(ii)*} \quad &d_{k} \Lambda-\Lambda d_{k+1}= -d^{c*}_{-n-k+j}, \\ 
    &\pa^*_{k+1}L-L\pa^*_{k}= -\iu\bpa_{-n-k+j}, &
    &\pa_{k} \Lambda-\Lambda \pa_{k+1}= -\iu\bpa^*_{-n-k+j}, \\
    &\bpa^*_{k+1}L-L\bpa^*_{k}= \iu\pa_{-n-k+j}, &
    &\bpa_{k} \Lambda-\Lambda \bpa_{k+1}= \iu\pa^*_{-n-k+j}, \\
    \textnormal{(iii)} \quad &\br{\mu}{L}=\br{\bmu}{L}=0, & 
    \textnormal{(iii)*} \quad &\br{\mu^*}{\Lambda}=\br{\bmu^*}{\Lambda}=0, \\
    \textnormal{(iv)} \quad &\br{\mu^*}{L}= \iu\bmu, & 
    \textnormal{(iv)*} \quad &\br{\mu}{\Lambda}= \iu\bmu^*, \\ 
    &\br{\bmu^*}{L}= -\iu\mu, & 
    &\br{\bmu}{\Lambda}= -\iu\mu^*. 
\end{align*}

%% file: HLC.tex
\section{Twisted harmonic forms and the Lefschetz condition for 
LCS manifolds}
\label{secHLC}

\subsection{The generalized Lefschetz duality for LCaK manifolds}

    Let $(M,J,g)$ be a compact LCaK manifold with the Lee form $\theta$.
    We define the following subspaces of 
    harmonic forms: 
    \[
    \mathcal{H}^{p,q}_k \coloneqq 
    \mathcal{H}^{p+q}_k \cap \mathcal{A}^{p,q} 
    \quad \text{and} \quad
    \mathcal{H}^j_{k,\mathcal{J}} \coloneqq 
    \{ \alpha \in \mathcal{H}^{j}_{k} \,\mid\, 
    \mathcal{J}\alpha \in \mathcal{H}^{j}_{k} \}. 
    \]
    We have the following inclusion: 
    \[
    \bigoplus_{p+q=j} \mathcal{H}^{p,q}_k
    \subset
    \mathcal{H}^{j}_{k,\mathcal{J}}. 
    \]

    These spaces have the following duality relations:  

    \begin{proposition}\label{dual}
        Let $(M,J,g)$ be a $2n$-dimensional compact 
        LCaK manifold with the Lee form $\theta$. 
        For all $0 \le j,p,q$,  
        the following duality relations hold:
        \begin{itemize}
        \setlength{\itemsep}{5pt}
            \item[(i)] \textnormal{(The complex conjugation).} 
            Taking the complex conjugation induces isomorphisms
            \[
            \mathcal{H}^{p,q}_{k} 
            \, \stackrel{\sim}{\longrightarrow} \,  
            \mathcal{H}^{q,p}_{k} \quad \text{and} \quad
            \mathcal{H}^{j}_{k,\mathcal{J}} 
            \, \stackrel{\sim}{\longrightarrow} \, 
            \mathcal{H}^{j}_{k,\mathcal{J}}. 
            \]
            \item[(ii)] \textnormal{(The Hodge duality).}
            The Hodge star operator $*$ induces isomorphisms
            \[
            \mathcal{H}^{p,q}_{k} 
            \, \stackrel{\sim}{\longrightarrow} \,  
            \mathcal{H}^{n-q,n-p}_{-k} \quad \text{and} \quad
            \mathcal{H}^{j}_{k,\mathcal{J}} 
            \, \stackrel{\sim}{\longrightarrow} \,  
            \mathcal{H}^{2n-j}_{-k,\mathcal{J}} . 
            \]
        \end{itemize}
    \end{proposition}

    \begin{proof}
        The duality relations (i) hold since $\Delta_k$ and $\mathcal{J}$ 
        are real operators. 
        The duality relations (ii) are obtained by Proposition \ref{Hodge} 
        and the commutativity between the operators
        $*$ and $\mathcal{J}$. 
    \end{proof}
    
    We next prove the generalized Lefschetz duality 
    for LCaK manifolds. 
    Firstly, 
    we prove the following 
    \begin{proposition}\label{LHinH}
        Let $(M,J,g)$ be a $2n$-dimensional 
        LCaK manifold with the Lee form $\theta$.  
        If $n+2k-j=0$,  
        we have $L \alpha \in \mathcal{H}^{j+2}_{k+1,\mathcal{J}}$ 
        and $\Lambda \alpha \in \mathcal{H}^{j-2}_{k-1,\mathcal{J}}$
        for all $\alpha \in \mathcal{H}^{j}_{k,\mathcal{J}}$.
    \end{proposition}

    \begin{proof}
        By Proposition \ref{KidforLCaK}, 
        we obtain
        \[
        d_{k+1} L \alpha = L d_k \alpha = 0
        \] 
        and 
        \begin{align*}
            d^*_{k+1} L \alpha 
            &= -d^c_{-n-k+j} \alpha + L d^*_{k} \alpha \\
            &= -\mathcal{J}^{-1} d_{k-(n+2k-j)} \mathcal{J} 
            \alpha +0 \\
            &= -\mathcal{J}^{-1} d_k (\mathcal{J} \alpha) \\
            &=0. 
        \end{align*} 
        Therefore we have $L \alpha \in \mathcal{H}^{j+2}_{k+1,\mathcal{J}}$. 
        By taking adjoints of all operators, 
        we have 
        $\Lambda \alpha \in \mathcal{H}^{j-2}_{k-1,\mathcal{J}}$. 
    \end{proof}

    Since $\mathcal{H}^{p,q}_k \subset \mathcal{H}^{p+q}_{k,\mathcal{J}}$, 
    we have the following 
    \begin{corollary}
        If $n+2k-(p+q)=0$,  
        we have $L \alpha \in \mathcal{H}^{p+1,q+1}_{k+1}$ 
        and $\Lambda \alpha \in \mathcal{H}^{p-1,q-1}_{k-1}$ 
        for all $\alpha \in \mathcal{H}^{p,q}_k$.
    \end{corollary}

    Combining these arguments with the Lefschetz decomposition on 
    the space of forms, 
    we obtain the following 

    \begin{theorem}
    [The hard Lefschetz duality for LCaK manifolds]
    \label{main}
        Let $(M,J,g)$ be a $2n$-dimensional 
        compact LCaK manifold with the Lee form $\theta$. 
        Then the operators 
        $\{ L, \Lambda, H\coloneqq \br{L}{\Lambda} \}$ 
        define finite 
        dimensional representations of $\mathfrak{sl}(2)$ on 
        \begin{equation}\label{slharm}
        \bigoplus_{p,q \ge 0} \mathcal{H}^{p,q}_{-(n-(p+q))/2} 
        \quad \text{and} \quad
        \bigoplus_{j \ge 0} \mathcal{H}^j_{-(n-j)/2, \mathcal{J}}. 
        \end{equation}
        Furthermore, 
        for all $0 \le p,q$ such that $j=p+q\le n$, 
        \[
        L^{n-j} \colon \mathcal{H}^{p,q}_{-(n-j)/2} 
        \, \stackrel{\sim}{\longrightarrow} \, 
        \mathcal{H}^{n-q,n-p}_{(n-j)/2}
        \]
        and 
        \[
        L^{n-j} \colon \mathcal{H}^j_{-(n-j)/2, \mathcal{J}} 
        \, \stackrel{\sim}{\longrightarrow} \,  
        \mathcal{H}^{2n-j}_{(n-j)/2, \mathcal{J}}
        \]
        are isomorphisms. 
    \end{theorem} 
    
    \begin{proof}
        As explained in Section \ref{secLef}, 
        $\{L, \Lambda, H\}$
        defines a representation of $\mathfrak{sl}(2)$ 
        on forms. 
        Since $n+2k-j=0$, 
        where $k=-(n-j)/2$, 
        $\{L, \Lambda, H\}$ defines a subrepresentation 
        on the spaces (\ref{slharm}) by 
        Proposition \ref{LHinH}. 
        This proves the first statement of the theorem. 
    
        We shall prove the second statement. 
        From the Lefschetz decomposition, 
        the operator 
        $L^{n-j}$ is injective as an operator acting 
        on $\mathcal{A}^{p+q}$. 
        Moreover, 
        by Proposition \ref{dual} 
        we have 
        $\dim \mathcal{H}^{p,q}_{-(n-j)/2} = 
        \dim \mathcal{H}^{n-q,n-p}_{(n-j)/2} < \infty$
        and 
        $\dim \mathcal{H}^j_{-(n-j)/2, \mathcal{J}} =
        \dim \mathcal{H}^{2n-j}_{(n-j)/2, \mathcal{J}} < \infty$. 
        So the maps are isomorphisms. 
    \end{proof}

    \begin{remark}
        The space of harmonic forms of bidegree $(p,q)$, 
        denoted by $\mathcal{H}^{p,q}$, 
        is defined 
        on almost K\"{a}hler manifolds in 
        \cite{MR4110721}. 
        Thus the space $\mathcal{H}^{p,q}_k$ can be regarded as a 
        generalization of the space $\mathcal{H}^{p,q}$ to LCaK manifolds. 
        On almost K\"{a}hler manifolds, 
        the relations 
        \[
        \mathcal{H}^{p,q}
        = \mathcal{H}^{p,q}_{\bpa} \cap \mathcal{H}^{p,q}_{\mu} 
        = \mathcal{H}^{p,q}_{\pa} \cap \mathcal{H}^{p,q}_{\bmu},  
        \]
        where
        $\mathcal{H}^{p,q}_{\delta}$ denotes a space 
        $\mathcal{A}^{p,q} \cap \ker \delta \cap \ker \delta^*$ 
        for an operator $\delta$, 
        hold \cite{MR4110721}. 
        However, 
        on LCaK manifolds, 
        the similar relations 
        do not hold in general. 
    \end{remark}

\subsection{Relations with the twisted cohomology 
and the hard Lefschetz condition for LCS manifolds}

    In this section, 
    we observe the relations of harmonic forms 
    $\mathcal{H}^{p,q}_{k}$ and 
    $\mathcal{H}^{j}_{k,\mathcal{J}}$
    with the twisted cohomology 
    of LCaK manifolds. 
    We define the Lefschetz map on the twisted 
    cohomology as follows: 
    \[
    [L] \colon H^j_{k}(M) \to H^{j+2}_{k+1}(M), 
    \quad 
    [\alpha] \mapsto [\omega \w \alpha]. 
    \]
    The map is well-defined since 
    the Leibniz rule (\ref{Leibniz}) holds. 
    We call the maps $[L]$ and $[L]^j$ also the Lefschetz maps. 
    The maps $[L]$ and $L$ are compatible with the following 
    natural inclusions: 
    \[
    \bigoplus_{p+q=j} \mathcal{H}^{p,q}_k
    \subset
    \mathcal{H}^{j}_{k,\mathcal{J}}
    \hookrightarrow
    H^{j}_k(M), 
    \]

    \begin{remark}
        On compact K\"{a}hler manifolds, 
        the above three spaces are isomorphic for each $j$. 
        In general, 
        however, 
        these inclusions are proper 
        even on compact CaK manifolds \cite{MR4110721}. 
        We give examples of compact LCaK manifolds with 
        the same property in Section \ref{secExamples}. 
    \end{remark}
    
    Considering the above, 
    we define the following condition 
    on compact LCS manifolds: 

    \begin{definition}[The hard Lefschetz condition for LCS manifolds]
    \label{HLCforLCS}
        Let $(M,\omega)$ be a $2n$-dimensional 
        compact LCS manifold with the Lee form $\theta$. 
        $(M,\omega)$ satisfies the
        \emph{hard Lefschetz condition} 
        (HLC) if the maps
        \[
        [L]^{n-j} \colon H^j_{-(n-j)/2}(M) 
        \to H^{2n-j}_{(n-j)/2}(M)
        \]
        are isomorphisms for all $1 \le j \le n$. 
    \end{definition}

    \begin{remark}
        If $(M,\omega)$ is a CS manifold, 
        we can take a global real-valued function $f$ on $M$ 
        such that $e^{-f}\omega$ is symplectic. 
        Then the above condition is equivalent to saying that 
        a symplectic manifold $(M,e^{-f}\omega)$ satisfies the HLC 
        in the classical sense. 
        In this sense, 
        the above condition is a generalization of the HLC for 
        symplectic manifolds to that for LCS manifolds. 
    \end{remark}

    \begin{remark}
        The HLC for LCS manifolds is suggested also in \cite{MR3746516}. 
        However, 
        an LCS manifold $(M,\omega)$ satisfies the HLC 
        defined in \cite{MR3746516} if and only if 
        $(M,\omega)$ is CS and 
        a symplectic manifold $(M,e^{-f}\omega)$ satisfies the HLC 
        in the classical sense. 
        So the HLC defined here is weaker than the one 
        defined in \cite{MR3746516}. 
    \end{remark}

    \begin{remark}
        The hard Lefschetz theorem 
        for Vaisman manifolds, 
        which are a special class of LCK manifolds, 
        is formulated in \cite{MR3885160}, 
        however, 
        the formulation has nothing to do 
        with the hard Lefschetz theorem formulated here 
        at least in the context of this paper. 
    \end{remark}

    Since the HLC for symplectic manifolds is always satisfied on 
    compact K\"{a}hler manifolds, 
    the following conjecture is naturally suggested: 

    \begin{conjecture}\label{HLTforLCK}
        Let $(M,J,g)$ be a $2n$-dimensional compact 
        LCK manifold with the Lee form $\theta$. 
        Then 
        $(M,J,g)$ satisfies the HLC. 
    \end{conjecture}

    \begin{remark}
        On some LCS manifolds, 
        the cohomology class of $\omega$ vanishes, 
        i.e.  
        $[\omega]=0 \in H^2_1(M)$. 
        For instance, 
        LCS manifolds constructed from 
        contact manifolds canonically 
        \cite{MR0418003}, 
        Vaisman manifolds, 
        and LCK manifolds with potential 
        \cite{ornea2023principles}. 
        In these cases, 
        the Lefschetz map $[L]$ is a zero map. 
        So the HLC for LCS is equivalent to saying that 
        the vanishings
        \[
        H^j_{-(n-j)/2}(M)=0
        \]
        hold for all $1\le j \le n-1$. 
    \end{remark}

%% file: Examples.tex
\section{Examples}\label{secExamples}

    As examples, 
    we give LCaK manifolds defined in Section $4$ of \cite{MR3763412}, 
    and compute the spaces defined in the previous sections 
    with their basises. 
    Furthermore, 
    we check that these manifolds do not satisfy the HLC. 
    The examples we give here 
    are solvmanifolds. 
    So before giving the examples, 
    we collect some facts about solvmanifolds and methods to 
    compute their twisted cohomologies. 

    Let $G$ be a simply connected Lie group 
    and $\mathfrak{g}$ be the associated Lie algebra. 
    The space $\bigwedge^{\bullet}\mathfrak{g}^*$ is identified 
    with left-invariant forms on $G$. 
    We take a left-invariant closed $1$-form $\theta \in \mathfrak{g}^*$. 
    Since the twisted differential $d_{\theta}$ 
    is left-invariant, 
    we obtain a complex 
    $(\bigwedge^{\bullet}\mathfrak{g}^*,d_{\theta})$ 
    called the \emph{twisted Chevalley-Eilenberg complex}. 
    Denote by $H^{\bullet}_{\theta}(\mathfrak{g^*})$ its cohomology groups. 
    A discrete subgroup $\Gamma$ of $G$ is called \emph{lattice} 
    if the quotient $\Gamma \backslash G$ is compact. 
    The following theorem gives a method to 
    compute the twisted cohomology of the manifold $\Gamma \backslash G$ 
    from the twisted Chevalley-Eilenberg complex 
    when $G$ is completely solvable, 
    i.e. 
    the endomorphisms $\ad_X$ of $\mathfrak{g}$ 
    have only real eigenvalues 
    for all $X \in \mathfrak{g}$. 

    \begin{theorem}[\cite{MR0124918}]
        Let $G$ be a simply connected and completely solvable Lie group 
        with a lattice $\Gamma$ and a left-invariant closed $1$-form $\theta$. 
        Then the natural inclusion $(\bigwedge^{\bullet}\mathfrak{g}^*,d_{\theta}) 
        \hookrightarrow 
        (\mathcal{A}^{\bullet}(\Gamma \backslash G),d_{\theta})$
        is a quasi-isomorphism, 
        i.e.  
        the inclusion induces an isomorphism in cohomology. 
   \end{theorem}

    From now on, 
    we only consider left-invariant forms. 
    So we write $\bigwedge^{j}\mathfrak{g}^*$ as 
    $\mathcal{A}^j$ by abuse of notation. 
    Let $m \ge 2$ be an integer. 
    We shall define a compact LCaK solvmanifold $(M_m,J,g)$ of dimension $2n=2m+2$. 
    
    We consider a Lie algebra 
    $\mathfrak{g}_m \coloneqq \R \ltimes_{A} \R^{2m+1}$, 
    where $A$ is given by the following diagonal matrix:
    \[
        A = \text{diag}  
        \Bigr(
            0, 
            \frac{1}{m}, 
            \frac{2}{m}, 
            \ldots, 
             1, 
             -\frac{1}{m}, 
             -\frac{2}{m}, 
             \ldots, 
             -1
        \Bigl). 
    \]
    Let $G_m$ be the associated simply connected Lie group. 
    It is proved in \cite{MR3763412} that the manifolds 
    $G_m$ are completely solvable, 
    and each of the manifolds admits at least one lattice $\Gamma_m$. 
    Thus we can compute the cohomologies of the compact manifolds  
    $M_m \coloneqq \Gamma_m \backslash G_m$ from the 
    associated Lie algebras. 
    We note that 
    the dimensions of $H^j(M_m)$ and $H^j_{\theta}(M_m)$ 
    are computed in \cite{MR3763412}.
    
    We take $(x^1, \cdots ,x^{m+1},y^1, \cdots ,y^{m+1})$ as 
    the standard basis of $\mathfrak{g}^*_m = (\R \ltimes_{A} \R^{2m+1})^*$. 
    We use the notation $x^{\alpha_1 \cdots \alpha_p}
    y^{\beta_1 \cdots \beta_q}$ for the form
    $x^{\alpha_1} \w \cdots \w x^{\alpha_p} 
    \w y^{\beta_1} \w \cdots \w y^{\beta_q}$. 
    The structure of $(\bigwedge^{\bullet}\mathfrak{g}^*_m,d)$ is 
    determined by the Leibniz rule and the following equations: 
    \begin{align*}
        dx^1 &=0, & dy^1 &=-x^1y^1, \\
        dx^2 &=0, & dy^2 &=(1/m)x^1y^2, \\
        dx^3 &=-(1/m)x^{13}, & dy^3 &=(2/m)x^1y^3,  \\ 
        dx^4 &=-(2/m)x^{14}, & dy^4 &=(3/m)x^1y^4, \\
        &\vdots & &\vdots \\
        dx^{m+1} &=-((m-1)/m)x^{1,m+1}, & dy^{m+1} &=x^1y^{m+1}. 
    \end{align*}
    
    We consider the following forms:  
    \[
    \omega \coloneqq \sum_{\alpha=1}^{m+1} x^{\alpha}y^{\alpha}, 
    \quad
    \theta\coloneqq \frac{1}{m} x^1. 
    \]
    It is straightforward to check that 
    the form $\omega$ endows $M_m$ with a left-invariant 
    LCS structure with the Lee form $\theta$. 
    Furthermore, 
    we give a left-invariant almost complex structure 
    $J$ to $M_m$ by considering 
    $z^{\alpha} \coloneqq x^{\alpha}+\iu y^{\alpha}$ as $(1,0)$-forms, 
    or equivalently, 
    by defining $\mathcal{J}$ as 
    \[
    \mathcal{J}x^{\alpha} = -y^{\alpha}, 
    \quad
    \mathcal{J}y^{\alpha} = x^{\alpha} 
    \]
    for all $1 \le \alpha \le m+1$. 
    Then $(x^1, \cdots ,x^{m+1},y^1, \cdots ,y^{m+1})$  is 
    an orthogonal basis of $\mathfrak{g}^*$ with respect to 
    the Riemannian metric $g(\cdot,\cdot) \coloneqq \omega(\cdot,J\cdot)$. 
    Given these structures, 
    $(M_m,J,g)$ is a compact LCaK manifold with the fundamental form 
    $\omega$ and the Lee form $\theta$. 
    We note that the manifolds $(M_m,J,g)$ are not LCK, 
    i.e. 
    $J$ is not integrable, 
    as the exterior derivative of 
    a $(1,0)$-form $z^2=x^2+\iu y^2$,  
    \begin{align*}
        dz^2 = d(x^2+\iu y^2) &= \frac{\iu}{m} x^1y^2 \\
        &= \frac{1}{4m}(z^1 z^2 - z^1 \conj{z}^2 
        + \conj{z}^1 z^2 - \conj{z}^1 \conj{z}^2), 
    \end{align*}
    is not contained in $\mathcal{A}^{2,0} \oplus \mathcal{A}^{1,1}$. 
    
    We shall compute the spaces 
    defined in the previous sections on these LCaK manifolds.  
    We fix the orthogonal basis of 
    $\bigwedge^j\mathfrak{g}^*_m$ as follows:
    \[
    \mathcal{B}^j \coloneqq 
    \{
    x^{\alpha_1 \cdots \alpha_p}y^{\beta_1 \cdots \beta_q} 
    \in 
    \sideset{}{^j}{\bigwedge}\mathfrak{g}^*_m
    \, \mid \, 
    1 \le \alpha_1<\cdots<\alpha_p \le m+1, 
    1 \le \beta_1<\cdots<\beta_q \le m+1
    \}. 
    \]
    We define the weight of each element of the basis 
    $W \colon \mathcal{B}^{\bullet} \to \R$
    as follows: 
    \begin{align*}
         W(x^1) &=0, & W(y^1) &=-m, \\ 
         W(x^2) &=0, & W(y^2) &=1, \\
         W(x^3) &=-1, & W(y^3) &=2, \\
         W(x^4) &=-2, & W(y^4) &=3, \\ 
         &\vdots & &\vdots \\ 
          W(x^{m+1}) &=-(m-1), & W(y^{m+1}) &=m, 
    \end{align*}
    \[
    W(x^{\alpha_1 \cdots \alpha_p}y^{\beta_1 \cdots \beta_q}) 
    =W(x^{\alpha_1})+ \cdots +W(x^{\alpha_p}) + 
    W(y^{\beta_1})+ \cdots +W(y^{\beta_q}), \quad 
    W(1)=0. 
    \]
    Using this, 
    we can calculate the twisted differential of 
    $\gamma \in \mathcal{B}^{\bullet}$ as follows: 
    \begin{equation}\label{weight}
        d_k(\gamma)= \frac{W(\gamma)-k}{m} x^1 \w \gamma. 
    \end{equation}
    
    We define the subspaces of $\mathcal{A}^j$ as follows: 
    \begin{align*}
    \mathcal{A}^j_k &\coloneqq 
    \left< \{ \gamma \in \mathcal{B}^j \, \mid \, W(\gamma)=k 
    \} \right>_{\C}, \\
    \mathcal{\hat{A}}^j_{k} &\coloneqq 
    \left< \{ \gamma \in \mathcal{B}^j \, \mid \, W(\gamma) \neq k, \, x^1 \w \gamma \neq 0 
    \} \right>_{\C}, 
    \end{align*}
    where $\left< S \right>_{\C}$ denotes a $\C$-vector space generated 
    by the elements of the set $S$. 
    Then we obtain the Hodge decomposition of the complex 
    $(\mathcal{A}^{\bullet},d_k)$: 
    \begin{proposition}
        The Hodge decomposition of the complex 
        $(\mathcal{A}^{\bullet},d_k)$ with respect to $g$ is 
        \[
        \mathcal{A}^j = \mathcal{A}^j_k \oplus 
        x^1 \mathcal{\hat{A}}^{j-1}_{k} \oplus 
        \mathcal{\hat{A}}^j_{k}. 
        \]
        More precisely, 
        \[
        \mathcal{H}^j_k = \mathcal{A}^j_k, \quad
        d_k \mathcal{A}^{j-1} = 
        x^1 \mathcal{\hat{A}}^{j-1}_{k} \quad \text{and} \quad
        d^*_k \mathcal{A}^{j+1} = 
        \mathcal{\hat{A}}^j_{k}, 
        \]
        where $x^1V$ denotes a vector space 
        $\{ x^1 \w v \, \mid \, v \in V \}$, 
        and $\mathcal{H}^j_k$ denotes the space 
        $\ker \Delta_k \cap \mathcal{A}^j$. 
    \end{proposition}

    \begin{proof}
        It is obvious that the decomposition is orthogonal. 
        From (\ref{weight}), 
        it follows that 
        the space $\mathcal{A}^j_k \oplus 
        x^1 \mathcal{\hat{A}}^{j-1}_{k}$
        is contained in $\ker d_k$. 
        Moreover, the map 
        \[
        d_k \colon \mathcal{\hat{A}}^{j-1}_{k} \to 
        x^1 \mathcal{\hat{A}}^{j-1}_{k}
        \]
        is an isomorphism. 
        So we have $d_k(\mathcal{A}^{j-1})=x^1 
        \mathcal{\hat{A}}^{j-1}_{k}$. 
        Therefore, the space 
        $\mathcal{A}^j_k
        \oplus \mathcal{\hat{A}}^j_{k}$ 
        is contained in $\ker d^*_k$. 
        So we have $d^*_k(\mathcal{A}^{j+1})= 
        \mathcal{\hat{A}}^{j}_{k}$ and 
        $\mathcal{H}^j_k =
        \mathcal{A}^j_k$. 
    \end{proof}

    From this proposition, 
    we obtain

    \begin{proposition}
        The LCS manifolds $(M_m,\omega)$ do not satisfy the HLC for all $m \ge 2$. 
    \end{proposition}
    \begin{proof}
    If $m=2m'$ is even, 
    a form $x^{m'+2} \in \mathcal{H}^1_{-m'}$ 
    is in the kernel of the Lefschetz operator, i.e. 
    $[L]^m([x^{m'+2}])=0 \in H^{2m+1}_{m'}(M_m)$. 
    If $m=2m'+1$ is odd, 
    a form $x^{2,m'+2} \in \mathcal{H}^2_{m'}$. 
    is in the kernel of the Lefschetz operator, i.e. 
    $[L]^{m-1}([x^{2,m'+2}])=0 \in H^{2m}_{m'}(M_m)$. 
    As a consequence, 
    the LCS manifolds $(M_m,\omega)$ do not satisfy the HLC for all $m \ge 2$. 
    \end{proof}
    
    Now we shall compute the basises of the spaces 
    $\mathcal{H}^j_{k}$, 
    $\mathcal{H}^j_{k,\mathcal{J}}$ 
    and $\mathcal{H}^{p,q}_k$ 
    where, 
    $n+2k-j=n+2k-(p+q)=0$ 
    for $m=2,3,4$. 
    We only consider the case $j \le n$ as the basises in higher degrees 
    are obtained by taking the Hodge dual. 
    We note that $\mathcal{H}^j_{k}$ is always zero if  
    $k$ is not integer or $j=0$. 

    \vspace{5pt}\paragraph{$\bullet \; \;  m=2$}\quad \vspace{5pt}
    
    The basis of each of the spaces are shown in Table \ref{m=2}. 
    We also have 
    $\mathcal{H}^{p,q}_k=\{0\}$ for all $p,q$. 
    The Lefschetz map 
    $[L]^2 \colon H^1_{-1}(M_2) \to H^5_{1}(M_2)$
    is as follows: 
    \[
    [x^3] \mapsto 0. 
    \]

    \vspace{5pt}\paragraph{$\bullet \; \;  m=3$}\quad \vspace{5pt}

    The basis of each of the spaces are shown in Table \ref{m=3}. 
    We also have 
    \begin{align*}
    \mathcal{H}^{1,1}_{-1} &= 
    \mathcal{H}^2_{\mathcal{J}} \cap \ker(\mathcal{J}-\id) \\
    &=\left< x^{13}+y^{13} \right>_{\C}, \\
    \mathcal{H}^{2,2}_{0} &= 
    \mathcal{H}^4_{\mathcal{J}} \cap \ker(\mathcal{J}-\id)
    \end{align*}
    \[
    =\left< x^{134}y^{4}+x^{4}y^{134}, x^{124}y^{3}+x^{3}y^{124}, 
    x^{123}y^{2}+x^{2}y^{123} \right>_{\C}, 
    \]
    and 
    $\mathcal{H}^{p,q}_k=\{0\}$ for all $p,q$ other than those above. 
    The Lefschetz map 
    $[L]^2 \colon H^2_{-1}(M_3) \to H^6_{1}(M_3)$ 
    on the elements which are not 
    represented by the elements in 
    $\mathcal{H}^2_{-1,\mathcal{J}}$ is as follows: 
    \begin{align*}
        [x^4y^2] &\mapsto 0, \\
        [x^{23}] &\mapsto 0.
    \end{align*}

    \vspace{5pt}\paragraph{$\bullet \; \;  m=4$}\quad \vspace{5pt}

    The basis of each of the spaces are shown in Table \ref{m=4}. 
    We also have $\mathcal{H}^{p,q}_k=\{0\}$ for all $p,q$. 
    The Lefschetz map 
    $[L]^4 \colon H^1_{-2}(M_4) \to H^9_{2}(M_4)$
    is as follows: 
    \[
    [x^4] \mapsto 0. 
    \]

    The Lefschetz map 
    $[L]^2 \colon H^3_{-1}(M_4) \to H^7_{1}(M_4)$ 
    on the elements which are not 
    represented by the elements in 
    $\mathcal{H}^3_{-1,\mathcal{J}}$ is as follows: 
    \begin{align*}
        [x^1y^{14}] &\mapsto 2[*(x^{45}y^5-x^{34}y^3-x^{24}y^2)], \\
        [x^{45}y^5] &\mapsto -2[*(x^1y^{14})], \\
        [x^{34}y^3] &\mapsto 2[*(x^1y^{14})], \\
        [x^{24}y^2] &\mapsto 2[*(x^1y^{14})], \\
        [x^{35}y^4] &\mapsto 0, \\
        [x^{25}y^3] &\mapsto 0.
    \end{align*}
    Therefore, 
    the image of this map is 
    \[
    \{ [\gamma] \in H^7_1(M_4) \, \mid \, 
    \gamma \in 
    \left< 
    *(x^{45}y^5-x^{34}y^3-x^{24}y^2), 
    *(x^1y^{14})
    \right>_{\C} \oplus 
    \mathcal{H}^7_{1,\mathcal{J}} \}. 
    \]

    \input{table_m=2}
    \input{table_m=3}
    \input{table_m=4}

%% file: table_m=2.tex
\begin{table}[h!]
    \centering
    \caption{$m=2$}
    \label{m=2}
    \begin{tabular}{c|c||c|c}
        \hline
        \multicolumn{2}{c||}{space} &$\mathcal{H}^1_{-1}$ & $\mathcal{H}^3_{0}$ \\ \hline \hline
        
        \multicolumn{2}{c||}{$\dim$} &$1$ & $4$ \\ \hline
        
        \multirow{2}{*}{basis} &&$x^{3}$& $x^{23}y^{2}$, $x^{1}y^{13}$ \\ \cline{2-4}

        & \multirow{1}{*}{in $\mathcal{H}^{\bullet}_{\bullet,\mathcal{J}}$} &&$x^{13}y^{2} \pm \iu x^{2}y^{13}$ \\ \hline
    \end{tabular}
\end{table}

%% file: table_m=3.tex
\begin{table}[h!]
    \centering
    \caption{$m=3$}
    \label{m=3}
    \begin{tabular}{c|c||c|c}
        \hline
        \multicolumn{2}{c||}{space} &$\mathcal{H}^2_{-1}$ & $\mathcal{H}^4_{0}$ \\ \hline \hline
        
        \multicolumn{2}{c||}{$\dim$}&$4$&$10$ \\ \hline
        
        \multirow{3}{*}{basis} 
        &&$x^4y^2$, $x^{23}$& $x^{1}y^{123}$, $x^{12}y^{14}$, $x^{234}y^{4}$, $x^{34}y^{23}$ \\ \cline{2-4}

        &\multirow{2}{*}{in $\mathcal{H}^{\bullet}_{\bullet,\mathcal{J}}$}
        &$x^{13} \pm y^{13}$&$x^{134}y^{4} \pm x^{4}y^{134}$, 
        $x^{124}y^{3} \pm x^{3}y^{124}$ \\
        &&&$x^{123}y^{2} \pm x^{2}y^{123}$ \\ \hline
    \end{tabular}
\end{table}

%% file: table_m=4.tex
\begin{table}[h!]
    \centering
    \caption{$m=4$}
    \label{m=4}
    \begin{tabular}{c|c||c|c|c}
        \hline
        \multicolumn{2}{c||}{space} &$\mathcal{H}^1_{-2}$ & $\mathcal{H}^3_{-1}$ & $\mathcal{H}^5_{0}$ \\ \hline \hline
        
        \multicolumn{2}{c||}{$\dim$}
        &$1$ & $12$ & $24$ \\ \hline
        
        \multirow{7}{*}{basis} 
        &&$x^{4}$ &$x^{1}y^{14}$, $x^{45}y^{5}$&$x^{12}y^{124}$, $x^{13}y^{134}$, $x^{13}y^{125}$, $x^{14}y^{135}$ \\
        &&        &$x^{34}y^{3}$, $x^{24}y^{2}$                              &$x^{15}y^{145}$, $x^{234}y^{23}$, $x^{235}y^{24}$, $x^{245}y^{34}$ \\
        &&        &$x^{35}y^{4}$, $x^{25}y^{3}$                              &$x^{245}y^{25}$, $x^{345}y^{35}$                                   \\ \cline{2-5}

        &\multirow{4}{*}{in $\mathcal{H}^{\bullet}_{\bullet,\mathcal{J}}$}
        &&$x^{3}y^{15} \pm \iu x^{15}y^{3}$&$x^{1234}y^{4} \pm \iu x^{4}y^{1234}$, $x^{1235}y^{5} \pm \iu x^{5}y^{1235}$ \\
        &&&$x^{2}y^{14} \pm \iu x^{14}y^{2}$&$x^{23}y^{125} \pm \iu x^{125}y^{23}$, $x^{23}y^{134} \pm \iu x^{134}y^{23}$ \\
        &&&$x^{123} \pm \iu y^{123}$&$x^{24}y^{135} \pm \iu x^{135}y^{24}$, $x^{25}y^{145} \pm \iu x^{145}y^{25}$ \\
        &&&&$x^{34}y^{145} \pm \iu x^{145}y^{34}$ \\ \hline
    \end{tabular}
\end{table}